\documentclass[10pt]{amsart}

\usepackage{amsmath,amssymb,amsfonts}
\usepackage{mathrsfs,latexsym,amsthm,enumerate}
\usepackage[all]{xy}

\usepackage{tikz}
\usepackage{verbatim}
\usepackage{graphicx}
\DeclareGraphicsExtensions{.png}

\newtheorem{theorem}{Theorem}[section]
\newtheorem{proposition}[theorem]{Proposition}
\newtheorem{example}[theorem]{Example}

\newtheorem{examples}[theorem]{Examples}
\newtheorem{lemma}[theorem]{Lemma}
\newtheorem{remark}[theorem]{Remark}

\newcommand{\dom}{\mathbf{d}}
\newcommand{\ran}{\mathbf{r}}

\input xypic

\title[Bass-Serre theory]{A categorical description of Bass-Serre theory}

\author{M.~V.~Lawson}
\address{Department of Mathematics
and the
Maxwell Institute for Mathematical Sciences,
Heriot-Watt University,
Riccarton,
Edinburgh~EH14~4AS,
UK}
\email{markl@ma.hw.ac.uk}

\author{A.~R.~Wallis}
\address{Department of Mathematics
and the
Maxwell Institute for Mathematical Sciences,
Heriot-Watt University,
Riccarton,
Edinburgh~EH14~4AS,
UK}
\email{arw5@hw.ac.uk}

\begin{document}
\begin{abstract} 
Self-similar group actions may be encoded by a class of left cancellative monoids called left Rees monoids,
a result obtained by combining pioneering work by  Perrot with later work by the first author.
Left Rees monoids that are also right cancellative are called Rees monoids.
Irreducible Rees monoids have the striking property that they are embedded in their universal groups
and these  universal groups are HNN extensions by a single stable letter.
In this paper, we generalize the above theory from monoids to categories.
We study the structure of left Rees categories and prove that Rees categories embed into their universal groups.
Furthermore, we show that from each graph of groups, we may construct a Rees category
and prove that the fundamental group of the former is the universal group of the latter.
In this way, Bass-Serre theory may be viewed as a special case of the general problem of embedding
categories into groupoids.\\

\end{abstract}

\maketitle

\thanks{The first author's research was partially supported by an EPSRC grant (EP/I033203/1)
and the second by an EPSRC Doctoral Training Account (EP/P504945/1). 
Some of the material in this paper appeared in the second author's PhD thesis \cite{AW}.}

%%%%%%%%%%%%%%%%%%%%%%%%%%%%%%%%%%
\section{Introduction and terminology}

In this paper, we generalize the theory developed in \cite{Lawson2008a, LW} from monoids to categories.
In particular, we shall show how to construct a  cancellative category from each graph of groups
in such a way that the universal groupoid of the category is the fundamental groupoid of the graph of groups.

Any undefined terms from category theory we use may be found in \cite{Mac}.
For us, categories are small and objects are replaced by identities;
thus we view them as `monoids with many identities'.
The elements of a category $C$ are called {\em arrows} and the set of identities of $C$ is denoted by $C_{o}$.
Each arrow $a$ has a {\em domain}, denoted by $\mathbf{d}(a)$, and a {\em codomain} denoted by $\mathbf{r}(a)$,
both of which are identities and $a = \mathbf{d}(a)a = a\mathbf{r}(a)$.
The product $ab$ exists 
if and only if $\ran (a) = \dom (b)$.\footnote{This is the reverse of the way that 
the first author usually treats category products but is the most natural one in the light of the applications we have in mind.}
Thus our products should be conceived thus
$$ e \stackrel{a}{\longrightarrow} f \stackrel{b}{\longrightarrow} i$$
where $e$, $f$ and $i$ are identities.
We shall sometimes write $\exists ab$ to mean that the product $ab$ exists in the category. 
Given identities $e$ and $f$, the set of arrows $eCf$ is called a {\em hom-set}
and $eCe$ is a monoid called the {\em local monoid at $e$}.
The group of units of the local monoid at $e$ is called the {\em local group at $e$}.
We say that arrows $x$ and $y$ are {\em parallel} if they belong to the same hom-set.
A category is called {\em left cancellative} if whenever $ax = ay$ we have that $x = y$.
We define {\em right cancellative} categories dually.
A {\em cancellative category} is one which is both left and right cancellative.
An arrow $a$ is {\em invertible} or an {\em isomorphism} if there is an arrow $a^{-1}$, called an {\em inverse} and perforce unique, 
such that $aa^{-1} = \mathbf{d}(a)$ and $a^{-1}a = \mathbf{r}(a)$.
An element $a \in C$ of a category is said to be an {\em atom} if it is not invertible and if $a = bc$ then either $b$ or $c$ is invertible.

\begin{lemma}\label{le: atoms_closed_under_groupoid_products} 
If $a$ is an atom in an arbitrary category, $g$ is invertible and $\exists ga$ then $ga$ is an atom.
\end{lemma}
\begin{proof} Suppose that $ga = bc$.
Then $a = (g^{-1}b)c$.
Thus $g^{-1}b$ or $c$ is invertible;
that is, $b$ or $c$ is invertible.
It follows that $ga$ is also an atom.
\end{proof}

A category in which every arrow is invertible is called a {\em groupoid}.
If a groupoid is just a disjoint union of its local groups then we say that it is {\em totally disconnected}.
The set of invertible elements of a category forms a groupoid with the same set of identities.
Later, we shall deal only with categories $C$ having the following property: any isomorphism belongs to a local monoid.
We shall say that such categories are {\em skeletal}.   
The following is well-known.

\begin{lemma} 
Every category is equivalent to a skeletal category.
\end{lemma}
\begin{proof} Let $C$ be a category.
Let $\{ e_{i} \colon i \in I \}$ be a transversal of the isomorphism classes of the identities.
Thus if $i \neq j$ then $e_{i}$ is not isomorphic to $e_{j}$ and every identity in $C$ is isomorphic to some $e_{i}$.
Let $C'$ be the full subcategory determined by the $e_{i}$.
Then $C'$ is a skeletal category equivalent to $C$.
\end{proof}

Given any directed graph $D$, we may construct the {\em free category on $D$} denoted by $D^{\ast}$.
This has one identity for each vertex and consists of all finite directed paths in the graph.
Multiplication is concatenation of paths.
The elements of $D^{\ast}$ will be written $a_{1} \cdot \ldots \cdot a_{m}$ where the $a_{i}$ are edges of the graph
such that $\stackrel{a_{1}}{\rightarrow} \stackrel{a_{2}}{\rightarrow} \ldots \stackrel{a_{m}}{\rightarrow}$.

With each category $C$, we may associate its {\em universal } or {\em fundamental groupoid} $U(C)$ \cite{MP}.
To construct this groupoid, first regard $C$ as a directed graph.
For each edge $e \stackrel{a}{\rightarrow} f$ attach a new edge $f \stackrel{a^{-1}}{\rightarrow} e$.
Form the free category $(C \cup C^{-1})^{\ast}$ on the set $C \cup C^{-1}$.
We denote elements of this category by $a_{1} \cdot \ldots \cdot a_{n}$ where $a_{i} \in C \cup C^{-1}$
and $\mathbf{r}(a_{i}) = \mathbf{d}(a_{i+1})$ for $i = 1,2,\ldots, n-1$.
Now define a congruence $\equiv$ on this free category generated by
$$a \cdot a^{-1} \equiv \mathbf{d}(a),
\quad
a^{-1} \cdot a \equiv \mathbf{r}(a),
\quad
a \cdot b \equiv ab \mbox{\rm \, if } \exists ab
$$  
where $a,b \in C$.
The first two identifications ensure that the quotient category is a groupoid and the third ensures that there
will be a functor from $C$ to the quotient.
This gives us our groupoid $U(C)$ with an associated functor $\iota \colon C \rightarrow U(G)$.
We prove that this has the correct universal property.
Let $\theta \colon C \rightarrow G$ be a functor to a groupoid.
We may extend $\theta$ to a function $\bar{\theta} \colon C \cup C^{-1} \rightarrow G$ 
where we define $\bar{\theta}(a) = \theta (a)$ and $\bar{\theta} (a^{-1}) = \theta (a)^{-1}$.
We may therefore extend $\bar{\theta}$ to a functor $\theta'$ from the free category $(C \cup C^{-1})^{\ast}$ to $G$.
Observe that if $x \equiv y$ then $\bar{\theta} (x) \equiv \bar{\theta} (y)$.
Thus we may define a functor $\Theta$ from $U(C)$ to $G$.
The uniqueness property comes from the fact that $U(C)$ is generated by $C$.

%%%%%%%%%%%%%%%%%%%%%%%%%%%%%%%%%%%%%%%%%%%%%%%%%%%%%%%%%%%%%%%%%%%%%%%%%%%%%%%%%%%%%%%%%%%%%%
\section{Levi categories}

In this section, we define a class of categories that contains the left Rees categories
which we are primarily interested in.

The following definition is generalized from semigroup theory \cite{MS}.
A category $C$ is said to be {\em equidivisible} if for every commutative square
$$\spreaddiagramrows{2pc}
\spreaddiagramcolumns{2pc}
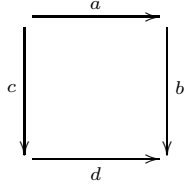
\diagram
& \dto_{c} \rto^{a}
& \dto^{b}
\\
& \rto_{d}
&
\enddiagram$$
we either have an arrow $u$ making the following diagram commute
$$\spreaddiagramrows{2pc}
\spreaddiagramcolumns{2pc}
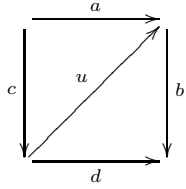
\diagram
& \dto_{c} \rto^{a}
& \dto^{b}
\\
& \rto_{d} \urto^{u}
&
\enddiagram$$
or one, $v$, making the following diagram commute
$$\spreaddiagramrows{2pc}
\spreaddiagramcolumns{2pc}
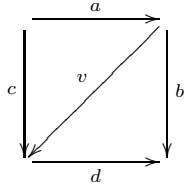
\diagram
& \dto_{c} \rto^{a}
& \dto^{b} \dlto_{v}
\\
& \rto_{d}
&
\enddiagram$$

A {\em length functor} is a functor $\lambda \colon C \rightarrow \mathbb{N}$ from a category $C$ to the additive monoid of natural numbers
satisfying the following conditions:
\begin{description}

\item[{\rm (LF1)}] If $xy$ is defined then $\lambda (xy) = \lambda (x) + \lambda (y)$.

\item[{\rm (LF2)}] $\lambda^{-1}(0)$ consists of all and only the invertible elements of $C$.

\item[{\rm (LF3)}] $\lambda^{-1}(1)$ consists of all and only the atoms of $C$.

\end{description}

Sometimes we shall write $\lambda_{C}$ if we wish to emphasize the fact that the length functor belongs to $C$.

A {\em Levi category} is an equidivisible category equipped with a length functor that contains at least one noninvertible element (to rule out groupoids).
These generalize the Levi monoids first introduced in \cite{LW}.

A {\em principal right ideal} in a category $C$ is a subset of the form $aC$ where $a \in C$.
We may similarly define {\em principal left ideals} and {\em principal ideals}. 
Consequently, Greens relations $\mathscr{L}$, $\mathscr{R}$, $\mathscr{H}$, $\mathscr{D}$,  and $\mathscr{J}$ can also be defined in categories.
Thus, for example, in the category $C$ we define $a \, \mathscr{L} \, b$ if and only if $Ca = Cb$.
Let $aC \subseteq bC$. 
We use the notation $[aC,bC]$ to mean the set of principal right ideals $xC$ such that $aC \subseteq xC \subseteq bC$.
The proof of the following is straightforward.

\begin{lemma}\label{le: greensrelations} Let $C$ be a category equipped with a length functor having $G$ as its groupoid of invertible elements.
\begin{enumerate}

\item $a \, \mathscr{L} \, b \Leftrightarrow Ga = Gb$. In particular, $\mathbf{r} (a) = \mathbf{r} (b)$.

\item $a \, \mathscr{R} \, b \Leftrightarrow aG = bG$. In particular, $\mathbf{d} (a) = \mathbf{d} (b)$.

\item $a \, \mathscr{J} \, b \Leftrightarrow GaG = GbG$. 

\end{enumerate}
\end{lemma}

Principal right ideals of the form $eC$ where $e$ is an identity are maximal such ideals
because if $eC \subseteq aC$ then $e = \dom (a)$ and $aC = eaC \subseteq eC$ giving $eC = aC$.
A principal right ideal $aC$ is said to be {\em submaximal} if $aC \neq \dom (a)C$ 
and there are no proper principal right ideals between $aC$ and $\dom(a)C$.

\begin{lemma}\label{le: atomic_decomposition} Let $C$ be a Levi category.
\begin{enumerate}

\item The element $a$ is an atom if and only if $aC$ is submaximal if and only if $Ca$ is submaximal.

\item Each non-invertible element $a$ of $C$ can be written $a = a'b$ where $a'$ is an atom.

\item Every non-invertible element of $C$ can be written as a product of atoms.

\end{enumerate}
\end{lemma}
\begin{proof} (1). Suppose that $a$ is an atom and that $aC \subseteq bC$.
Then $a = bc$.
But $a$ is an atom and so either $b$ is invertible or $c$ is invertible.
Suppose that $c$ is invertible then $aC = bC$.
Suppose that $b$ is invertible then $bC = \dom (b) C$.
We have proved that $aC$ is submaximal.

Conversely, suppose that $aC$ is submaximal.
Let $a = bc$.
Then $aC \subseteq bC$.
It follows that either $aC = bC$ or $bC = \dom (b)C$.
If the latter occurs then $b$ is invertible.
If the former occurs then $\lambda (a) = \lambda (b)$.
It follows that $\lambda (c) = 0$ and so $c$ is invertible.

(2). Let $a$ be a non-invertible element.
If $a$ is an atom then we are done.
If not, then $a = a_{1}b_{1}$ for some $b_{1}$ and $a_{1}$ where neither $a_{1}$ nor $b_{1}$ are invertible.
Observe that $\lambda (a_{1}) < \lambda (a)$.
If $a_{1}$ is an atom then we are done, otherwise we may write 
$a_{1} = a_{2}b_{2}$ again where neither $a_{2}$ and $b_{2}$ are atoms.
Observe that $\lambda (a_{2}) < \lambda (a_{1})$.
This process can only continue in a finite number of steps and will end with
$a = a_{n}b'$ for some $b'$ where $a_{n}$ is an atom.

(3). If $a$ is an atom then there is nothing to prove.
Otherwise by (2), we may write $a = a_{1}b_{1}$ where $a_{1}$ is an atom and $\lambda (b_{1}) = \lambda (a) - 1$.
If $b_{1}$ were invertible then $a$ would have been an atom.
If $b_{1}$ is an atom then we are done.
Otherwise, we may repeat the above procedure with $b_{1}$.
\end{proof}

In categories equipped with length functors, we can write every element in terms of atoms and invertible elements.
The obvious next question is what kind of uniqueness we can expect.
Under the additional assumption of equidivisibility, we can retrieve a kind of uniqueness.

\begin{lemma}\label{le: uniqueness} Let $C$ be a Levi category
and suppose that 
$$x = a_{1} \ldots a_{m} = b_{1} \ldots b_{n}$$ 
where the $a_{i}$ and $b_{j}$ are atoms.

\begin{enumerate}

\item $m = n$.

\item There are invertible elements $g_{1}, \ldots, g_{n-1}$ such that
$$a_{1} = b_{1}g_{1}, \quad
a_{2} = g_{1}^{-1}b_{2}g_{2}, \quad
\ldots
\quad
a_{n} = g_{n-1}^{-1}b_{n}.$$ 
This data is best presented by means of the following {\em interleaving diagram}.

$$\spreaddiagramrows{1pc}
\spreaddiagramcolumns{1pc}
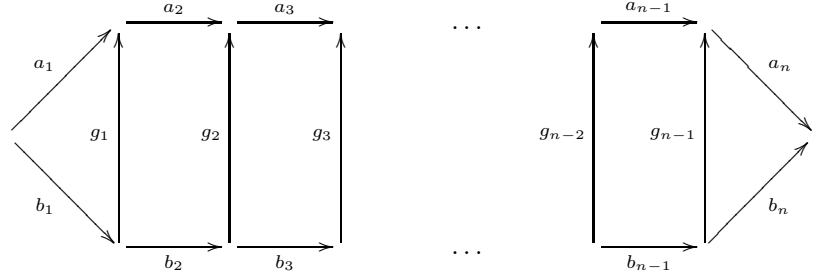
\diagram
& 
& \rto^{a_{2}} 
& \rto^{a_{3}}
&
& \ldots
& \rto^{a_{n-1}}
& \drto^{a_{n}}
&
\\
& \urto^{a_{1}} \drto_{b_{1}}  
& 
&
&
&
&
&
&
\\
& 
& \uuto^{g_{1}} \rto_{b_{2}} 
& \uuto^{g_{2}} \rto_{b_{3}}
& \uuto^{g_{3}} 
& \ldots
& \uuto^{g_{n-2}} \rto_{b_{n-1}}
& \uuto^{g_{n-1}} \urto_{b_{n}}
&
\enddiagram$$

\item If $C$ is in addition skeletal then $a_{i}$ is parallel to $b_{i}$ for $i = 1,\ldots, m$.

\end{enumerate}
\end{lemma}
\begin{proof} (1). This is immediate from the properties of length functors.

(2). We bracket as follows
$$a_{1}(a_{2}  \ldots a_{m}) = b_{1}(b_{1} \ldots b_{m}).$$
By equidivisibility, 
$a_{1} = b_{1}u$ and $b_{2} \ldots b_{m} = u a_{2} \ldots a_{m}$ for some $u$ 
or 
$b_{1} = a_{1}v$ and $a_{2} \ldots a_{m} = vb_{2} \ldots b_{m}$ for some $v$.
In either case, $u$ and $v$ are invertible since both $a_{1}$ and $b_{1}$ are atoms using the length function.
Thus 
$a_{1} = b_{1}g_{1}$, where $v = g_{1}$ is an isomorphism,
and 
$b_{2} \ldots b_{m} = g_{1}a_{2} \ldots a_{m}$.

We now repeat this procedure bracketing thus
$$b_{2}(b_{3} \ldots b_{m}) = g_{1}a_{2}(a_{3} \ldots a_{m}).$$
By the same argument as above, we get that
$g_{1}a_{2} = b_{2}g_{2}$ for some isomorphism $g_{2}$
and 
$b_{3} \ldots b_{m} = g_{2}a_{3} \ldots a_{m}$.

The process continues and we obtain the result.

(3). The result is immediate from the assumption that the category is also skeletal.
\end{proof}

We shall be interested in {\em left cancellative} Levi categories and ultimately those that are also skeletal.
But we shall start with a slightly different definition and show that it is equivalent to this one.
A category $C$ is said to be {\em right rigid} if $aC \cap bC \neq \emptyset$ implies that $aC \subseteq bC$ or $bC \subseteq aC$.
A {\em left Rees category} is a left cancellative, right rigid category in which each principal right ideal is properly contained 
in only finitely many distinct principal right ideals.

\begin{example} {\em Free categories are left Rees categories: 
the atoms are the edges; 
the length functor simply counts the number of edges in a path;
the groupoid of invertible elements is trivial.}
\end{example}

\begin{lemma} Let $C$ be a left cancellative category with groupoid of invertible elements $G$.
\begin{enumerate}

\item  $a \, \mathscr{R} \, b$ if and only if $aG = bG$.

\item $C$ is a right rigid category if and only if it  is equidivisible.

\item If $e = xy$ is an identity then $x$ is invertible with inverse $y$.

\end{enumerate}
\end{lemma}
\begin{proof} (1). Only one direction needs proving.
Suppose that  $a \, \mathscr{R} \, b$.
Then $a = bx$ and $b = ay$ for some $x,y \in C$.
Thus $a = ayx$ and $b = bxy$.
By left cancellation both $xy$ and $yx$ are identities and so $x$ is invertible with inverse $y$.

(2). Only one direction needs proving. Suppose that $ab = cd$.
Then $aC \cap cC \neq \emptyset$.
Without loss of generality, we may suppose by right rigidity that $aC \subseteq cC$.
Thus $a = cu$
But then $cub = cd$ and so $ub = d$, as required. 

(3). We have that $xyx = x$ and by left cancellation this shows that $yx$ is an identity and so $x$
is invertible with inverse $y$.
\end{proof}

In the light of the lemma above, the following says that
left cancellative equidivisible categories equipped with length functors are left Rees categories.

\begin{lemma} Let $C$ be a right rigid, left cancellative category equipped with a length functor $\lambda$.
Then $C$ is a left Rees category.
\end{lemma}
\begin{proof} Let $a \in C$ be any element with $e = \dom (a)$.
We need to prove that the set $[aC,eC]$ is finite.
Let $bC \in [aC,eC]$.
Then $a = bx$ for some $x$ and so $\lambda (a) \geq \lambda (b)$.
There is therefore an upper bound on the lengths of those elements $b$ such that $aC \subseteq bC$. 
Let $b_{1}C,b_{2}C \in [aC,eC]$ and suppose that $\lambda (b_{1}) = \lambda (b_{2})$.
By right rigidity, we may assume, without loss of generality that $b_{1}C \subseteq b_{2}C$.
Thus $b_{1} = b_{2}x$ for some $x$.
But $b_{1}$ and $b_{2}$ have the same length and so $x$ must have length zero.
It follows that $x$ is invertible.
Hence $b_{1}C = b_{2}C$.
It follows that the set $[aC,eC]$ is finite, as claimed.
\end{proof}

We shall now prove the converse to the above result.
This involves proving that every left Rees category is equipped with a length functor.
Let $a \in C$ be an element of a left Rees category.
We shall define the {\em length}, $\lambda (a)$, of $a$.
Put $\mathbf{d}(a) = e$.
By assumption, the set $[aC,eC]$ is finite and linearly ordered.
The proofs of the following are straightforward.

\begin{lemma} Let $C$ be a left Rees category.
\begin{enumerate} 

\item The set $[aC,eC]$ contains one element if and only if $a$ is invertible.

\item The set $[aC,eC]$ contains two elements if and only if $a$ is an atom.

\end{enumerate}
\end{lemma}

Let $a \in C$ where $aC \subseteq eC$ where $e$ is an identity.
Define 
$$\lambda (a) = \left| [aC,eC] \right| -1.$$

\begin{lemma}\label{le: length_lemma} Let $C$ be a left Rees category and let $a,b \in C$ such that $ab$ is defined.
\begin{enumerate}

\item $a[bC,\dom (b)C] = [abC, aC]$.

\item $[abC,\mathbf{d}(a)C] = a[bC,\mathbf{d}(b)C] \cup [aC,\mathbf{d}(a)C]$.

\item If $bC \subseteq cC, dC \subseteq C$ and $acC = adC$ then $cC = dC$.

\end{enumerate}
\end{lemma}
\begin{proof} 
(1). Let $xC \in [bC,\dom (b)C]$. Then $bC \subseteq xC \subseteq \dom (b)C$.
It is immediate that $abC \subseteq axC \subseteq aC$.
Thus the lefthand side is contained in the righthand side.
Now let $abC \subseteq yC \subseteq aC$.
Then $y = au$ for some $u \in C$.
Thus $yC = auC = a(uC)$.
Now $ab = yv$ for some $v$.
Hence $ab = auv$ and so by left cancellation, we have that $b = uv$.
Thus $bC \subseteq uC \subseteq \dom (b)C$.

(2).  Let $c = ab$.
Then $cC \subseteq aC \subseteq \dom (a)C$.
Let $xC \in [cC,\mathbf{d}(c)C]$.
Since this set is linearly ordered either $xC \subseteq aC$ or $aC \subseteq xC$.
It follows that the lefthand side is contained in the righthand side.
The proof of the reverse containment is immediate.

(3). This is immediate by left cancellation.
\end{proof}

The proof of the following is now immediate.

\begin{lemma} Let $C$ be a left Rees category.
\begin{enumerate}

\item $\lambda (ab) = \lambda (a) + \lambda (b)$.

\item $\lambda (a) = 0$ if and only if $a$ is invertible.

\item $\lambda (a) = 1$ if and only if $a$ is an atom.

\end{enumerate}
\end{lemma}

We have therefore proved the following.

\begin{proposition}\label{prop: fred} 
A left cancellative, right rigid category is a left Rees category if and only if it is equipped with a length functor.
In other words, the left Rees categories are precisely the left cancellative Levi categories.
\end{proposition}

The following result is worth noting here.

\begin{lemma} Let $C$ be a left Rees category.
Then each local monoid of $C$ is a left Rees monoid.
\end{lemma}
\begin{proof} Let $S = eCe$ be a local monoid.
It is immediate that it is a left cancellative monoid.
Let $aS \cap bS \neq \emptyset$ where $a,b \in S$.
Then clearly $aC \cap bC \neq \emptyset$.
Without loss of generality, it follows that $aC \subseteq bC$.
Thus $a = bc$ for some $c \in C$.
Now $eae = a$ and $b = ebe$ so that $a = b(ece)$ and $c = ece$.
We have therefore proved that $aS \subseteq bS$.
Thus $S$ is right rigid.
We denote by $[aS,bS]$ the obvious set of principal right ideals {\em in} $S$.
The set $[aS,eS]$ is linearly ordered.
We shall prove that it is finite.
Let $aS \subseteq b_{1}S \subset b_{2}S \subseteq eS$ where $b_{1},b_{2} \in S$.
We shall prove that $b_{1}C \neq b_{2}C$.
Suppose on the contrary that $b_{1}C = b_{2}C$.
Then $b_{1} = b_{2}g$ for some invertible element $g$.
then $b_{1} = b_{2}(ege)$ and $ege = g$.
Thus $g$ is an invertible element in $S$ and so $b_{1}S = b_{2}S$, a contradiction.
Since, by assumption, the set $[aC,bC]$ is finite it follows that the set $[aS,bS]$ is finite.
\end{proof}

\begin{remark} 
{\em Let $C$ be a left Rees category. 
It is important to observe that although $S = eCe$ is a left Rees monoid,
its length function need not be the restriction of the one in $C$. 
The length of the element $a \in S$, viewed as an element of the monoid $S$, 
is defined to be one less than the number of elements of $[aS,eS]$.
The length of the element $a \in S$ viewed as an element of $C$ is defined to be one less than the
number of elements in $[aC,eC]$.
But the latter set may contain more elements than the former.
Thus $\lambda_{S}(a) \leq \lambda_{C}(a)$. 
Consider the following example.
Let $C$ be the free category defined by the following directed graph
$$
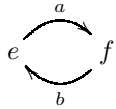
\xymatrix{e \ar@/^1pc/[r]^{a} & f \ar@/^1pc/[l]^{b}}$$
The local monoid at $e$ is just $S = (ab)^{\ast}$, the free monoid on one generator $ab$.
Here $\lambda_{S}(ab) = 1$ but $\lambda_{C}(ab) = 2$ since both $a$ and $b$ are atoms in $C$.}
\end{remark}

\begin{remark} {\em It can be shown that condition (LF3) is not needed in the definition of a length functor.
If Levi categories are defined with respect to this apparently weaker definition, a length functor may be easily
constructed that satisfies this condition using the analogue of Lemma~\ref{le: length_lemma}.
This was carried out explicitly in the monoid case described in \cite{LW}, 
and the proof there may be easily extended to the category case.}
\end{remark}

%%%%%%%%%%%%%%%%%%%%%%%%%%%%%%%%%%%%%%%%%%%%%%%%%%%%%%%%%%%%%%%%%%%%%%%%%%%%%%%%%%%%%%%%%%%
\section{Constructing Levi categories}

The goal of this section is to show that each Levi category
 is isomorphic to a tensor category over what we call a bimodule: 
that is, sets on which a given groupoid acts on the left and the right in such a  way that the two actions associate.

Let $G$ be a groupoid and let $X$ be a set equipped with two functions 
$$G_{0} \stackrel{\mathbf{s}}{\longleftarrow} X \stackrel{\mathbf{t}}{\longrightarrow} G_{o}.$$
We suppose that there is a left groupoid action $G \times X \rightarrow X$ and a right groupoid action
$X \times G \rightarrow X$ such that the two actions associate meaning $(gx)h = g(xh)$ when defined.
We write $\exists gx$ and $\exists xg$ if the actions are defined.
Observe that $\exists gx$ iff $\ran (g) = \mathbf{s}(x)$ and $\exists xg$ iff $\mathbf{t}(x) = \dom (g)$.
We call the structure $(G,X,G)$ a {\em bimodule} or a {\em $(G,G)$-bimodule}.
If whenever $\exists xg$ and $xg = x$ we have that $g$ is an identity, 
then we say the action is {\em right free}.
A bimodule which is right free is called a {\em covering bimodule}.
We define {\em left free} dually.
A bimodule which is both left and right free is said to be {\em bifree}.
We define {\em homomorphisms} and {\em isomorphisms} between $(G,G)$-bimodules in the usual way.

Our first result shows that bimodules arise naturally from Levi categories.
The proof is straightforward; in particular, 
the fact that the actions are well-defined follows from 
Lemma~\ref{le: atoms_closed_under_groupoid_products}. 

\begin{lemma}\label{le: category_to_bimodule}
Let $C$ be a Levi category.
Denote by $X$ the set of all atoms of $C$ equipped with the maps $\mathbf{d}, \mathbf{r} \colon X \rightarrow C_{0}$.
Denote by $G$ the groupoid of invertible elements of $C$.
Define a bimodule $(G,X,G)$ where the left and right actions are defined
via multiplication in $C$ when defined.
We obtain a covering bimodule if $C$ is left cancellative and a bifree bimodule if $C$ is cancellative
\end{lemma}

We call $(G,X,G)$ constructed as in the above lemma, the {\em bimodule associated with $C$}
or the {\em bimodule of atoms of $C$}.

Our goal now is to show that from each bimodule we may construct a Levi category.
Our tool for this will be tensor products and the construction of a suitable tensor algebra analogous to the ones defined in module theory; 
see Chapter~6 of \cite{Street}, for example.
We recall the key definitions and results we need first.

Let $G$ be a groupoid that acts on the set $X$ on the right and the set $Y$ on the left.
We consider the set $X \ast Y$ consisting of those pairs $(x,y)$ where $\mathbf{t}(x) = \mathbf{s}(y)$. 
A function $\alpha \colon X \ast Y \rightarrow Z$ to a set $Z$ is called a {\em bi-map} or a {\em 2-map}
if $\alpha (xg,y) = (x,gy)$ for all $(xg,y) \in X \ast Y$ where $g \in G$.
We may construct a universal such bimap $\lambda \colon X \ast Y \rightarrow X \otimes Y$
in the usual way \cite{EZ}.
However, there is a simplification in the theory due to the fact that we are acting by means of a groupoid.
The element $x \otimes y$ in $X \otimes Y$ is the equivalence class of $(x,y) \in X \ast Y$ under the relation $\sim$
where $(x,y) \sim (x',y')$ if and only if $(x',y') = (xg^{-1},gy)$ for some $g \in G$.
Observe that we may define $\mathbf{s}(x \otimes y) = \mathbf{s}(x)$ and $\mathbf{t}(x \otimes y) = \mathbf{t}(y)$
unambiguously.

Suppose now that $X$ is a $(G,G)$-bimodule.
We may therefore define the tensor product $X \otimes X$ as a set.
This set is equipped with maps $\mathbf{s}, \mathbf{t} \colon X \otimes X \rightarrow G_{o}$.
We define $g(x \otimes y) = gx \otimes y$ and $(x \otimes y)g = x \otimes yg$ when this makes sense.
Observe that $x \otimes y = x' \otimes y'$ implies that $gx \otimes y = gx' \otimes y'$, and dually.
It follows that $X \otimes X$ is a also a bimodule.
Put $X^{\otimes 2} = X \otimes X$.
More generally, we may define  $X^{\otimes n}$ for all $n \geq 1$ using $n$-maps, and we define $X^{\otimes 0} = G$
where $G$ acts on itself by multiplication on the left and right.
The proof of the following lemma is almost immediate from the definition and the
fact that we are acting by a groupoid.

\begin{lemma}\label{le: equality_of_tensors}
Let $n \geq 2$.
Then
$$x_{1} \otimes \ldots \otimes x_{n} = y_{1} \otimes \ldots \otimes y_{n}$$
if and only if 
there are elements $g_{1}, \ldots, g_{n-1} \in G$
such that
$y_{1} = x_{1}g_{1}$,
$y_{2} = g_{1}^{-1}x_{2}g_{2}$,
$y_{3} = g_{2}^{-1}x_{3}g_{3}$,
$\ldots$,
$y_{n} = g_{n-1}^{-1}x_{n}$.
\end{lemma}

Define
$$\mathsf{T}(X) = \bigcup_{n = 0}^{\infty} X^{\otimes n}.$$
We shall call this the {\em tensor category} associated with the bimodule $(G,X,G)$.
Observe that we may regard $X$ as a subset of $\mathsf{T}(X)$. 
The justification for this terminology will follow from (1) below.

\begin{theorem}\label{the: first_theorem}\mbox{}
\begin{enumerate}

\item $\mathsf{T}(X)$ is a Levi category whose associated bimodule is $(G,X,G)$.

\item The category $\mathsf{T}(X)$ is left cancellative if and only if $(G,X,G)$ is right free, and dually.

\item Let $C$ be a category whose groupoid of invertible elements is $G$.
Regard $C$ as a $(G,G)$-bimodule under left and right multiplication.
Let $\theta \colon X \rightarrow C$ be any bimodule morphism to $C$.
Then there is a unique functor $\Theta \colon \mathsf{T}(X) \rightarrow C$ extending $\theta$.

\item Every Levi category is isomorphic to the tensor category of its associated bimodule.

\end{enumerate}
\end{theorem}
\begin{proof} (1). The identities of the category are the same as the identities of $G$.
The element $x_{1} \otimes \ldots \otimes x_{n}$ has domain $\mathbf{d}(x_{1})$ and codomain $\ran (x_{n})$.
Multiplication is tensoring of sequences that begin and end in the right places and left and right actions by elements of $G$.
We define $\lambda (g) = 0$ where $g \in G$ and $\lambda (x_{1} \otimes \ldots \otimes x_{n}) = n$.
Formally, we are using the fact that there is a canonical isomorphism 
$$X^{\otimes p} \otimes X^{\otimes q} \cong X^{\otimes (p+q)}.$$
The proof of equidivisibility is essentially the same as that of Proposition~5.6 of \cite{Lawson2008a}.
The elements of length 0 are precisely the elements of $G$ and so the invertible elements;
the elements of length 1 are precisely the elements of $X$.

(2). Suppose that the category is left cancellative and that $xg = x = x\dom (g)$ in the bimodule.
But this can also be interpreted as a product in the category and so $g = \dom (g)$, as required.
Conversely, suppose that the bimodule is right free.
Let $\mathbf{x} \otimes \mathbf{y} = \mathbf{x} \otimes \mathbf{z}$.
From Lemma~\ref{le: equality_of_tensors} and the fact that lengths match,
we have that $(\mathbf{x},\mathbf{y}) = (\mathbf{x}g, g^{-1}\mathbf{z})$ for some $g \in G$.
But using the fact that the action is right free, we get that $g$ is an identity and so $\mathbf{y} = \mathbf{z}$,
as required.

(3). Define $\Theta (g) = g$ when $g \in G$ and 
$\Theta (x_{1} \otimes x_{2} \otimes \ldots \otimes x_{n}) 
=
\theta (x_{1}) \otimes \theta (x_{2}) \otimes \ldots \otimes \theta (x_{n})$.
This is well-defined by Lemma~\ref{le: equality_of_tensors}.
It is routine to check that this defines a functor.

(4). This now follows from (3) above, part (3) of Lemma~\ref{le: atomic_decomposition}, 
and part (2) of Lemma~\ref{le: uniqueness}.
\end{proof}

\begin{remark}
{\em It follows by the above theorem that left Rees categories are described by covering bimodules.}
\end{remark}

%%%%%%%%%%%%%%%%%%%%%%%%%%%%%%%%%%%%%%%%%%%%%%%%%%%%%%%%%%%%%%%%%%%%%%%%%%%%%%%%%%%%%%%%
\section{Left Rees categories}

Left cancellative Levi categories are what we call left Rees categories.
Such Levi categories admit an alternative description from the tensor category description of the previous section.
Here, we shall describe the structure of arbitrary left Rees categories in terms of free categories
using Zappa-Sz\'ep products generalized to categories; see \cite{Brin}.
This will show that they can be regarded as the categories associated with self-similar groupoid actions.
The material in this section can be regarded as a special case of \cite{Lawson2008b}.
However, we have included it for the sake of completeness.

Let $G$ be a groupoid with set of identities $G_{0}$ and let $C$ be a category with set of identities $C_{o}$.
We shall suppose that there is a bijection between $G_{0}$ and $C_{o}$ and, to simplify notation, we shall identify these two sets.
Denote by $G \ast C$ the set of pairs
$(g,x)$ such that $\ran (g) = \dom (x)$. 
We shall picture such pairs as follows:
$$\spreaddiagramrows{2pc}
\spreaddiagramcolumns{2pc}
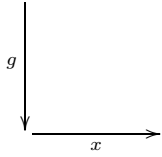
\diagram
&\dto_{g} 
& 
\\
&\rto_{x}
&
\enddiagram$$
We suppose that there is a function 
$$G \ast C \rightarrow C \text{ denoted by } (g,x) \mapsto g \cdot x,$$
which gives a left action of $G$ on $C$
and a function
$$G \ast C \rightarrow G \text{ denoted by } (g,x) \mapsto g|_{x},$$
which gives a right action of $C$ on $G$
such that these two functions satisfy the following conditions:
\begin{description}
\item[{\rm (C1)}] $\dom (g \cdot x) = \dom (g)$.
\item[{\rm (C2)}] $\ran (g \cdot x) = \dom (g|_{x})$.
\item[{\rm (C3)}] $\ran (x) = \ran (g|_{x})$.
\end{description}
This information is summarized by the following diagram
$$\spreaddiagramrows{2pc}
\spreaddiagramcolumns{2pc}
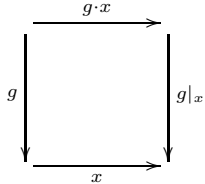
\diagram
& \dto_{g} \rto^{g \cdot x}
& \dto^{g|_{x}}
\\
& \rto_{x}
&
\enddiagram$$
We also require that the following axioms be satisfied:
\begin{description}

\item[{\rm (SS1)}] $\dom (x) \cdot x = x$. Observe that this is the action, not the category product.

\item[{\rm (SS2)}] If $gh$ is defined then $(gh) \cdot x = g \cdot (h \cdot x)$.

\item[{\rm (SS3)}] $\dom (g) = g \cdot \ran (g)$.

\item[{\rm (SS4)}] $\dom (x)|_{x} = \ran (x)$.

\item[{\rm (SS5)}] $g|_{\ran (g) } = g$.

\item[{\rm (SS6)}] If $xy$ is defined and $\ran (g) = \dom (x)$ then $g|_{xy} = (g|_{x})|_{y}$.

\item[{\rm (SS7)}] If $gh$ is defined and $\ran (h) = \dom (x)$ then $(gh)|_{x} = g|_{h \cdot x} h|_{x}$.

\item[{\rm (SS8)}] If $xy$ is defined and $\ran (g) = \dom (x)$ then $g \cdot (xy) = (g \cdot x)(g|_{x} \cdot y)$.
\end{description}
If there are maps $(g,x) \mapsto g \cdot x$ and $(g,x) \mapsto g|_{x}$ satisfying (C1)--(C3) and (SS1)--(SS8) then we say that there is a {\em self-similar action of $G$ on $C$.}
Put 
$$C \bowtie G = \{(x,g) \in C \times G \colon \: \ran (x) = \dom (g) \}.$$
We represent $(x,g)$ by the diagram
$$\spreaddiagramrows{2pc}
\spreaddiagramcolumns{2pc}
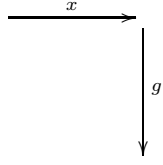
\diagram
&\rto^{x}
&\dto^{g}
\\
&
&
&
\enddiagram$$
Given elements $(x,g)$ and $(y,h)$ satisfying
$\ran (g) = \dom (y)$
we then have the following diagram
$$\spreaddiagramrows{2pc}
\spreaddiagramcolumns{2pc}
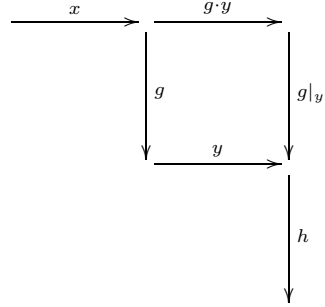
\diagram
&\rto^{x}
&\dto^{g} \rto^{g \cdot y}
& \dto^{g|_{y}}
\\
&
& \rto^{y} 
& \dto^{h}
\\
&
&
&
\enddiagram$$
Completing the square, as shown, 
enables us to define a partial binary operation on $C \bowtie G$ by 
$$(x,g)(y,h) = (x(g \cdot y), g|_{y}h).$$

\begin{lemma} Let $G$ be a groupoid having a self-similar action on the category $C$.
\begin{enumerate}

\item If $y$ is an invertible element of $C$ then so too is $g \cdot y$.

\item If $g \in G$ and $x \in X$ is an atom then $g \cdot x$ is an atom.

\item If $C$ is a left Rees category with length function $\lambda$ then $\lambda (g \cdot x) = \lambda (x)$.

\end{enumerate}
\end{lemma}
\begin{proof} (1).  Suppose that $e \stackrel{y}{\rightarrow} f$.
We have that $g \cdot e = \dom (g)$, by (SS3), and $e = yy^{-1}$.
Thus $\dom (g) = (g \cdot y)(g|_{y} \cdot y^{-1})$.
It follows that $g \cdot y$ is invertible with inverse $g|_{y} \cdot y^{-1}$.

(2). Let $g \cdot x = uv$, where $u,v \in C$.
By axioms (SS1), (SS2) and (SS8), we have that 
$$x = (g^{-1} \cdot u)(g^{-1}|_{u} \cdot v).$$
By assumption, $x$ is an atom and so at least one of the elements in the product is invertible.
Suppose that $g^{-1} \cdot u$ is invertible.
Then by our result above $g \cdot (g^{-1} \cdot u) = u$ is invertible.
Suppose now that $g^{-1}|_{u} \cdot v$ is invertible then we may deduce that $v$ is invertible.
We have therefore proved that $g \cdot x$ is an atom.

(3). Write $x = x_{1} \ldots x_{n}$ a product of atoms where $n = \lambda (x)$.
Now use (SS8).
\end{proof}

\begin{proposition} Let $G$ be a groupoid having a self-similar action on the category $C$.
\begin{enumerate}

\item $C \bowtie G$ is a category.

\item $C \bowtie G$ contains copies $C'$ and $G'$ of $C$ and $G$ respectively
such that each element of $C \bowtie G$ can be written as a product
of a unique element from $C'$ followed by a unique element from $G'$.

\item If $C$ is left cancellative then so too is $C \bowtie G$.

\item If $C$ is left cancellative then the set of invertible elements of $C \bowtie G$ 
consists of all those elements $(x,g)$ where $x$ is invertible in $C$

\item If $C$ is left cancellative then the set of atoms in $C \bowtie G$ 
consists of all those elements $(x,g)$ where $x$ is an atom in $C$.

\item If $C$ is left cancellative and right rigid then so too is $C \bowtie G$.

\item If $C$ is a left Rees category then so too is $C \bowtie G$.

\end{enumerate}
\end{proposition}
\begin{proof}
(1) Define $\dom (x,g) = (\dom (x), \dom (x))$ and $\ran (x,g) = (\ran (g), \ran (g))$.
The condition for the existence of $(x,g)(y,h)$ is that $\ran (x,g) = \dom (y,h)$.
Axioms (C1),(C2) and (C3) then guarantee the existence of $(x(g \cdot y), g|_{y}h)$.
We therefore have a partially defined multiplication.
We next locate the identities.
Suppose that $(u,a)$ is an element such that if $(u,a)(x,g)$ is defined then $(u,a)(x,g) = (x,g)$.
Now $(u,a)(\ran (a),\ran (a))$ is defined.
We deduce that
$\ran (a) = u(a \cdot \ran (a))$ 
and
$\ran (a) = a|_{\ran (a)} \ran (a)$.
By (SS5), we have that $a = \ran (a)$ and by (SS3) that $a \cdot \ran (a) = \dom (a)$. 
Thus $(u,a) = (\ran (a), \ran (a))$.
By (SS1) and (SS4), we deduce that the identities are the elements of the form $(e,e)$ where $e \in C_{o} = G_{o}$.
Observe that $\dom [(x,g)(y,h)] = \dom (x,g)$ and $\ran [(x,g)(y,h)] = \ran (y,h)$.
It remains only to prove associativity.

Suppose first that 
$[(x,g)(y,h)](z,k)$
exists.
The product $(x,g)(y,h)$ exists and so we have the following diagram
$$\spreaddiagramrows{2pc}
\spreaddiagramcolumns{2pc}
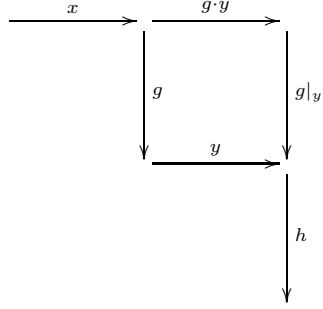
\diagram
&\rto^{x}
&\dto^{g} \rto^{g \cdot y}
& \dto^{g|_{y}}
\\
&
& \rto^{y} 
& \dto^{h}
\\
&
&
&
\enddiagram$$
similarly 
$[(x,g)(y,h)](z,k)$ exists
and so we have the following diagram
$$\spreaddiagramrows{2pc}
\spreaddiagramcolumns{2pc}
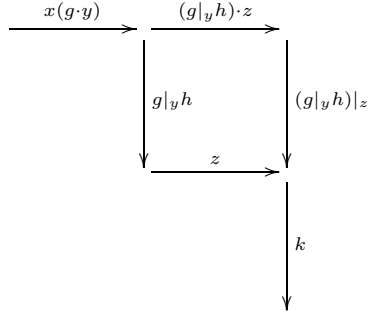
\diagram
& \rto^{x(g \cdot y)}
& \dto^{g|_{y}h} \rto^{( g|_{y} h) \cdot z}
& \dto^{(g|_{y}h )|_{z} }
\\
&
&  \rto^{z} 
&  \dto^{k}
\\
&
&
&
\enddiagram$$
resulting in the product
$$(x(g \cdot y)[(g|_{y}h) \cdot z],(g|_{y}h)|_{z}k).$$
By assumption, 
$x(g \cdot y)[(g|_{y}h) \cdot z]$ exists and so
$(g \cdot y)[(g|_{y}h) \cdot z]$ exists.
We now use (SS8) and (SS4) and (SS7), 
to get that $y(h \cdot z)$ exists
and we use (SS7) and (SS6) to show that
$$(g|_{y}h)|_{z}k = g|_{y(h \cdot z)}h|_{z}k.$$
By (SS2),
$$x(g \cdot y)[(g|_{y}h) \cdot z]
=
(x(g \cdot y))(g|_{y} \cdot (h \cdot z)).$$
It now follows that
$$(y,h)(z,k) = (y(h \cdot z),h|_{z}k)$$
exists.
It also follows that 
$(x,g)[(y,h)(z,k)]$ exists
and is equal to 
$$[(x,g)(y,h)](z,k).$$
Next suppose that 
$$(x,g)[(y,h)(z,k)]$$
exists.
This multiplies out to give
$(x[g \cdot (y(h \cdot z))], g|_{y(h \cdot z)}h|_{z}k)$.
By (SS6) and (SS7) we get that
$$g|_{y(h \cdot z)}h|_{z}k
=
(g|_{y}h)|_{z}k,$$
and by (SS8) and (SS2) we get that
$x[g \cdot (y(h \cdot z))]
=
x(g \cdot y)[(g|_{y}h) \cdot z]$.
This completes the proof that 
$C \bowtie G$ is a category.

(2) Define $\iota_{C} \colon \: C \rightarrow C \bowtie G$ by $\iota_{C}(x) = (x,\ran (x))$.
This is well-defined.
Suppose that $xy$ exists.
Then in particular
$\ran (x) = \dom (y)$.
It is easy to check using (SS4) and (SS1)
that $\iota_{C}(x)\iota_{C}(y) = \iota_{C}(xy)$.
In fact, $\iota_{C}(x)\iota_{C}(y)$ exists iff $xy$ exists.
Thus the categories $C$ and $C'$ are isomorphic.

Now define $\iota_{G} \colon \: G \rightarrow C \bowtie G$
by $\iota_{G}(g) = (\dom (g)  ,g)$.
Then once again the categories $G$ and $G'$ are isomorphic.

Finally, pick an arbitrary non-zero element $(x,g)$.
Then $\ran (x) = \dom (g)$.
We may write
$(x,g) = (x,\ran (x))(\dom (g) ,g)$
using the fact that 
$\ran (x) \cdot \dom (g) = \dom (g)$ by (SS1)
and
$\ran (x)|_{\dom (g)} = \dom (g)$ by (SS4).

(3) Suppose that $C$ is left cancellative.
We prove that $C \bowtie G$ is left cancellative.
Suppose that $(x,g)(y,h) = (x,g)(z,k)$.
Then $x(g \cdot y) = x(g \cdot z)$ and $g|_{y}h = g|_{z}k$.
By left cancellation in $C$ it follows that $g \cdot y = g \cdot z$
and by (SS1) we deduce that $y = z$.
Hence $h = k$.
We have therefore proved that $(y,h) = (z,k)$, as required.

(4) We know by (3), that the resulting category is left cancellative.
Suppose that $(x,g)$ is invertible.
Then there is an element $(y,h)$ such that $(\dom (x), \dom (x)) = (x,g)(y,h)$.
In particular, $\dom (x) = x (g \cdot y)$ and so $x$ is invertible.
Conversely, if $x$ is invertible, it can be verified that
$$(x,g)^{-1} = (g^{-1} \cdot x^{-1}, (g|_{g^{-1} \cdot x^{-1}})^{-1}).$$

(5). Suppose that $x$ is an atom.
Let $(x,g) = (u,h)(v,k)$.
Then $x = u(h \cdot v)$ and $g = h|_{v}k$.
If $u$ is invertible then $(u,h)$ is invertible, whereas if $h \cdot v$ is invertible then $v$ is invertible and so
$(v,k)$ is invertible.
It follows that $(x,g)$ is an atom.
To prove the converse, suppose that $(x,g)$ is an atom.
It is immediate that $x$ is not invertible.
Suppose that $x$ is not an atom.
Then we may write $x = uv$ where neither $u$ nor $v$ is invertible.
But then $(x, \ran (x)) = (u, \ran (u))(v, \ran (x))$
and this leads to a non-trivial factorization of $(x,g)$, which is a contradiction.

(6) Suppose now that $C$ is left cancellative and right rigid.
By (3), we know that $C \bowtie G$ is left cancellative 
so it only remains to be proved that $C \bowtie G$ is right rigid.
Suppose that
$$(x,g)(y,h) = (u,k)(v,l)$$
From the definition of the product it follows that
$x(g \cdot y) = u(k \cdot v)$
and
$g|_{y}h = k|_{v}l$.
From the first equation, we see that $xC \cap uC \neq \emptyset$.
Without loss of generality, suppose that $x = uw$.
Then by left cancellation $w(g \cdot y) = k \cdot v$.
Observe that $k^{-1} \cdot (k \cdot v)$ is defined
and so $k^{-1} \cdot (w(g \cdot y))$ is defined by (SS2).
Thus by (SS8), $k^{-1} \cdot w$ is defined.
It is now easy to check that
$$(x,g) = (u,k)(k^{-1} \cdot w, (k|_{k^{-1} \cdot w})^{-1}g).$$

(7) Let $C$ be a left Rees category.
It is enough to prove that $C \bowtie G$ is equipped with a length function.
Let the length function on $C$ be denoted by $\lambda$.
Let $(x,g) \in C \bowtie G$.
Define $\mu (x,g) = \lambda (x)$.
By the above, $\mu (x,g) = 0$ if and only if $(x,g)$ is invertible,
and $\mu (x,g) = 1$ if and only if $x$ is an atom.
The fact that $\mu$ is a functor follows from the fact that $\lambda (g \cdot y) = \lambda (y)$.
\end{proof}

We call $C \bowtie G$ the {\em Zappa-Sz\'ep product} of the category $C$ by the groupoid $G$.

Let $C$ be a left Rees category.
A transversal of the generators of the submaximal principal right ideals is called a {\em basis} for the category.
From Section~3, a basis is therefore a subset of the set of atoms of $C$. 

\begin{theorem} A category is a left Rees category if and only if it is isomorphic to the
Zappa-Sz\'ep product of a free category by a groupoid.
\end{theorem}
\begin{proof} We shall sketch the proof.
Let $C$ be the left Rees category.
Choose a basis $X$ for $C$.
Every element of $C$ can be written uniquely
as a product of elements of $X$ followed by an invertible element.
The subcategory $X^{\ast}$ generated by $X$ is free.
Thus $C = X^{\ast}G$.
The Zappa-Sz\'ep product representation then readily follows.
\end{proof}

The following theorem describes the precise circumstances under which Zappa-Sz\'ep products arise.

\begin{theorem} Let $A$ be a category with subcategories $B$ and $C$.
We suppose that $A_{o} = B_{o} = C_{o}$ and that $C = AB$, uniquely.
Then $C$ is isomorphic to the Zappa-Sz\'ep product $A \bowtie B$.
\end{theorem}
\begin{proof} We sketch out the proof.
Suppose that $ba$ is defined where $b \in B$ and $a \in A$.
Then $ba = a'b'$ for uniquely determined elements $a' \in A$ and $b' \in B$.
We define
$$ba = (b \cdot a) b|_{a}.$$
It is immediate that axioms (C1), (C2) and (C3) hold.
The equality $a = \dom (a)a$ yields the axioms (SS1) and (SS4);
the equality $b = b\ran (b)$ yields axioms (SS3) and (SS5);
the equality $b_{1}(b_{2}a) = (b_{1}b_{2})a$ yields axioms (SS2) and (SS7);
the equality $b(a_{1}a_{2}) = (ba_{1})a_{2}$ yields axioms (SS6) and (SS8).
We may therefore construct the category $A \bowtie B$.
We define a map $C \rightarrow A \bowtie B$ by $c \mapsto (a,b)$ if $c = ab$.
It is now straightforward to check that this determines an isomorphism of categories.
\end{proof}

The following result shows that the theory in this paper can be regarded as
a generalization of the theory of free categories.

\begin{theorem} The free categories are precisely the left Rees categories with a trivial groupoid of
invertible elements and precisely the equidivisible categories with length functors having trivial groupoids of invertible elements.
\end{theorem}
\begin{proof} We need only prove that an 
equidivisible category with length functor having a trivial groupoid of invertible elements
is left cancellative.
But this essentially follows from Lemma~\ref{le: uniqueness}. 
\end{proof}

\begin{remark} {\em Cancellative equidivisible monoids with trivial groups of units do not have to be free;
see, for example, Example~1.8 of Chapter~5 of \cite{Lallement}.
This example also suggests that studying equidivisible categories with more general kinds of length functors may be interesting.}
\end{remark}

%%%%%%%%%%%%%%%%%%%%%%%%%%%%%%%%%%%%%%%%%%%%%%%%%%%%%%%%%%%%%%%%%%%%%%%%%%%%%%%%%%%%%%%%%%%%%%
\section{Diagrams of partial homomorphisms}

{\em From now on, the groupoids involved in bimodules will always be totally disconnected.}

We showed in the previous section, that covering bimodules are the building blocks on left Rees monoids.
In this section, we want to describe how covering bimodules may be constructed.
This generalizes a construction due to Nekeshevych in the case where the acting groupoid is actually a group \cite{N1}.

Let $D$ be a directed graph.
An edge $x$ from the vertex $e$ to the vertex $f$ will be written $e \stackrel{x}{\rightarrow} f$.
With each vertex $e$ of $D$ we associate a group $G_{e}$, called the {\em vertex group},
and with each edge $e \stackrel{x}{\rightarrow} f$,
we associate a surjective homomorphism $\phi_{x} \colon (G_{e})^{+}_{x} \rightarrow (G_{f})^{-}_{x}$
where $(G_{e})^{+}_{x} \leq G_{e}$ and $(G_{f})^{-}_{x} \leq G_{f}$.
In other words, with each edge $e \stackrel{x}{\rightarrow} f$,
we associate a partial homomorphism $\phi_{x}$ from $G_{e}$ to $G_{f}$.
We call this structure a {\em diagram of partial homomorphisms}.
If all the $\phi_{x}$ are {\em isomorphisms} then we shall speak of a  {\em diagram of partial isomorphisms}.
For brevity, we shall say that $D$ is the diagram of partial homomorphisms though of course it is defined
only by the totality of data.

Let $D_{1}$ and $D_{2}$ be two diagrams of partial homomorphisms having the same vertex sets,
and edge sets that differ only in labelling and the same vertex groups.
We say these two diagrams of partial homomorphisms are {\em conjugate},
if for each edge $e \stackrel{x}{\rightarrow} f$ in $D_{1}$ and corresponding edge $e \stackrel{y}{\rightarrow} f$ in $D_{2}$ 
there are inner automorphisms
$\alpha_{x,y} \colon G_{e} \rightarrow G_{e}$ 
and
$\beta_{x,y} \colon G_{f} \rightarrow G_{f}$ 
such that
$\alpha_{x,y}( (G_{e})_{x}^{+}  ) = (G_{e})^{+}_{y}$
and
$\beta_{x,y}( (G_{f})_{x}^{-}  ) = (G_{f})^{-}_{y}$
and
$\beta_{x,y} \phi_{x} = \phi_{y}\alpha_{x,y}$.

\begin{remark}
{\em The above definition could be generalized a little by assuming only that the
underlying graphs were isomorphic.}
\end{remark}

The goal of this section is to prove that covering bimodules and diagrams of partial homomorphisms
are different ways of describing the same object.

Let $(G,X,G)$ be a covering bimodule where $G = \bigcup_{e \in V} G_{e}$.
Define a relation $\mathscr{C}$ on $X$ by $x \, \mathscr{C} \, \, y$ if and only if
$y = gxh$ for some $g,h \in G$.\footnote{This relation was used by Paul Cohn whence the choice of notation.}
Observe that because $G$ is a disjoint union of groups if $x \, \mathscr{C} \, \, y$
then $x$ and $y$ are parallel.
Clearly, $\mathscr{C}$ is an equivalence relation.

\begin{lemma}\label{le: bimodule_to_diagram} Let $(G,X,G)$ be a covering bimodule where $G = \bigcup_{e \in V} G_{e}$.
\begin{enumerate}

\item With each transversal of the $\mathscr{C}$-classes, we may associate a diagram of partial homomorphisms.

\item Different transversals yield conjugate diagrams of partial homomorphisms.

\item The bimodule is bifree if and only if the associated diagram is a diagram of partial isomorphisms.

\end{enumerate}
\end{lemma}
\begin{proof} (1). Choose a transversal $E$ of the $\mathscr{C}$-classes.
Define the directed graph $D$ to have as vertices the set of identities of $G$
and edges the elements of $E$.
The group associated with the vertex $e$ is the group $G_{e}$.
Consider now the edge $e \stackrel{x}{\rightarrow} f$.
We define
$$(G_{e})^{+}_{x} = \{g \in G_{e} \colon gx = xh \mbox{ for some } h \in G_{f} \}$$
and 
$$(G_{f})^{-}_{x} = \{h \in G_{f} \colon gx = xh \mbox{ for some } g \in G_{e} \}$$
and
$\phi_{x}(g) = h$ if $g \in (G_{e})^{+}_{x}$ and $gx = xh$ where $h \in G_{f}$.
In a covering bimodule the right action is free and so $\phi_{x}$ is a function and,
in fact, a homomorphism.
We have therefore constructed a diagram of partial homomorphisms.

(2). Let $E'$ be another transversal.
We denote the elements of $E$ by $x_{i}$ where $i \in I$ and the elements of $E'$ by $y_{i}$ where $i \in I$
and assume that $x_{i} \, \mathscr{C} \, y_{i}$.
Choose $g_{i},h_{i} \in G$ such that $x_{i} = g_{i}y_{i}h_{i}$.
Let $e \stackrel{x_{i}}{\rightarrow} f$
and
$g \in (G_{e})_{x_{i}}^{+}$.
Then
$$(g_{i}^{-1}gg_{i})y_{i} = y_{i}(h_{i} \phi_{x_{i}}(g)h_{i}^{-1}).$$
Thus
$$h_{i}\phi_{x_{i}}(g)h_{i}^{-1} = \phi_{y_{i}}(g_{i}^{-1}gg_{i}).$$
Define inner automorphisms 
$\beta_{i}(-) = h_{i} - h_{i}^{-1}$ 
and
$\alpha_{i}(-) = g_{i}^{-1} - g_{i}$
where $\alpha_{i} \colon G_{e} \rightarrow G_{e}$
and
$\beta_{i} \colon G_{f} \rightarrow G_{f}$.
We therefore have $\beta_{i} \phi_{x_{i}} = \phi_{y_{i}} \alpha_{i}$
and
$\alpha_{i}((G_{e})_{x_{i}}^{+}) = (G_{e})_{y_{i}}^{+}$
and
$\beta_{i}((G_{f})_{x_{i}}^{-}) = (G_{f})_{y_{i}}^{-}$.
Thus the two diagrams are conjugate

(3). Suppose that the bimodule is bifree.
By definition $gx = x \phi_{x}(g)$ when $g \in G_{x}^{+}$.
Suppose that $g_{1}x = g_{2}x$.
Then $x = (g_{1}^{-1}g_{2})x$ and so by left freeness we have that $g_{1}^{-1}g_{2}$ is an identity and so
$g_{1} = g_{2}$.
Thus $\phi_{x}$ is injective and so an isomorphism.
The proof of the converse is straightforward.
\end{proof}

We now show how to go in the opposite direction.
Let $D$ be a diagram of partial homomorphisms.
Let $G = \bigcup_{e \in V} G_{e}$ be the disjoint union of the vertex groups regarded as a groupoid.
Denote the set of edges by $E$.
Let $G \ast E \ast G$ be the set of triples $(g,x,h)$
where $e \stackrel{x}{\rightarrow} f$ and $g \in G_{e}$ and $h \in G_{f}$.
We define a relation $\equiv$ on the set $G \ast E \ast G$ as follows:
$(g_{1},x_{1},h_{1}) \equiv (g_{2},x_{2},h_{2})$ 
if and only if
$x_{1} = x_{2} = x$, say,
$g_{2}^{-1}g_{1} \in (G_{e})_{x}^{+}$ and
$\phi_{x}(g_{2}^{-1}g_{1}) = h_{2}h_{1}^{-1}$.

\begin{lemma}\label{le: diagram_to_bimodule} With the above definition, $\equiv$ is an equivalence relation.
Denote the $\equiv$-class containing $(g,x,h)$ by $[g,x,h]$.
Then we get a covering bimodule $\mathsf{B}(D)$ when we define
$g[g_{1},x,h_{1}] = [gg_{1},x,h_{1}]$ when $\exists gg_{1}$ 
and
$[g_{1},x,h_{1}]h = [g_{1},x,h_{1}h]$ when $\exists h_{1}h$.
The diagram of partial homomorphisms associated with this covering module is conjugate to $D$.
\end{lemma} 
\begin{proof} We shall just prove the last part since the other proofs are routine.
We choose the transversal $[e,x,f]$ where $e \stackrel{x}{\rightarrow} f$ is an edge of the diagram $D$.
It is now easy to check that we get back exactly the digram of partial homomorphisms we started with.
\end{proof}

\begin{lemma}\label{le: fred} Let $(G,X,G)$ be a covering bimodule, let $E$ be a transversal of the $\mathscr{C}$-classes
and let $D$ be the associated diagram of partial homomomorphisms.
Then the bimodule $X$ is isomorphic to the bimodule $\mathsf{B}(D)$.
\end{lemma}
\begin{proof} Let $x \in X$. Then $x \, \mathscr{C} \, y$ for a unique $y \in E$.
By definition, $x = g_{1}yh_{1}$ for some $g_{1},h_{1} \in G$.
Suppose also that $x = g_{2}yh_{2}$.
Then $g_{2}^{-1}g_{1} y = y h_{2}h_{1}^{-1}$.
It follows that $(g_{1},y,h_{1}) \equiv (g_{2}, y, h_{2})$.
We may therefore define a function
$\theta \colon X \rightarrow \mathsf{B}(D)$ by $\theta (x) = [g_{1},x,h_{1}]$.
It remains to show that this is a bijection and an isomorphism of bimodules
both of which are now straightforward.
\end{proof}

We summarize the results of this section in the following theorem.

\begin{theorem} With each diagram of partial homomorphisms $D$ we may associate a covering bimodule $\mathsf{B}(D)$
and every covering bimodule is isomorphic to one of this form.
Diagrams of partial isomorphisms correspond to bifree covering bimodules.
\end{theorem}

%%%%%%%%%%%%%%%%%%%%%%%%%%%%%%%%%%%%%%%%%%%%%%%%%%%%%%%%%%%%%%%%%%%%%%%%%%%%%%%%%%%%%%%%%%%%%%%
\section{Presentations of skeletal left Rees categories}

In this section, we shall prove that every skeletal left Rees category has a presentation of a particular form.
This presentation is then the final link in showing that the theory of graphs of groups is a special
case of the theory of skeletal left Rees categories.

Let $C$ be a skeletal left Rees category.
Then for each atom $x \in eCf$, we have as before the following definitions:
$$(G_{e})_{x}^{+} = \{ g \in G_{e} \colon gx = xh \mbox{ for some } h \in G_{f} \},$$
and
$$(G_{f})_{x}^{-} = \{ h \in G_{f} \colon gx = xh \mbox{ for some } g \in G_{e} \},$$
and
$$\phi_{x} (g) = h \mbox{ if } g \in (G_{e})_{x}^{+} \mbox{ and } h \in (G_{f})_{x}^{-} \mbox{ and } gx = xh.$$
We have that $(G_{e})_{x}^{+} \leq G_{e}$ and $(G_{f})_{x}^{-} \leq G_{f}$ and $\phi_{x}$ is a surjective homomorphism.

\begin{lemma}\label{le: right_cancellative_category} Let $C$ be a skeletal left Rees category.
Then $C$ is right cancellative if and only if all the homomorphisms $\phi_{x}$ defined above are also injective and so in fact isomorphisms.
\end{lemma}
\begin{proof} If $C$ is right cancellative, it is immediate that all the homomorphisms $\phi_{x}$ are injective.
We prove the converse using Section~5.
Choose a basis for $C$ so that $C = X^{\ast}G$ where $X^{\ast}$ is a free category and $G$ the groupoid of isomorphisms.
Let $a,b,c \in C$ such that $ab = cb$.
We shall prove that $a = c$.
We have that $a = xg$, $b = yh$ and $c = zk$ where $x,y,z \in X^{\ast}$ and $g,h,k \in G$.
Thus
$$x (g \cdot y)g|_{y} h = z(k \cdot y)k|_{y} h.$$
From length considerations, $x = z$.
We also have that $g|_{y} = k|_{y}$.
It remains only to prove that $g = k$ and we are done.
We know that $gy = ky$.
Thus $(g^{-1}k)y = y$.
Therefore, to prove our result it is enough to prove that $gy = y$ implies that $g$ is an identity.
If $y$ is an atom then the result is immediate.
We assume the result holds for all elements $y$ whose length is at most $n$ and prove it for those
elements of length $n+1$.
Let $y$ have length $n+1$.
Let $y = uv$ where $v$ has length 1.
Then $gy = guv = (g \cdot u) (g|_{u} \cdot v)g|_{uv} = y = uv$.
Thus $g|_{y} = g|_{uv} = \ran (g)$.
Also $u = g \cdot u$ and $v = g|_{u} \cdot v$.
We have that $g|_{u}v = v g|_{uv}$.
It follows that $g|_{u}$ is an identity.
We then have $gu = u$ and so $g$ is an identity, as required.
\end{proof}

The proof of the following is immediate from the definitions and Lemma~\ref{le: greensrelations}.

\begin{lemma} Let $C$ be a skeletal left Rees category.
If $x$ and $y$ are atoms then $x \, \mathscr{J} \, y$ if and only if $x \, \mathscr{C} \, y$.
\end{lemma}

The following two results will enable us to refine the way that we choose a basis for a left Rees category.
We refer the reader to Section~5 for the structure theory of left Rees categories which we use here.

\begin{lemma} Let $C$ be a left Rees category and let $X$ be a transversal of the generators of
the submaximal principal right ideals so that $C = X^{\ast}G$.
Let $x,y \in X$.
Then $x \, \mathscr{J} \, y$ if and only if  $y = g \cdot x$.
\end{lemma}
\begin{proof} Suppose that $y = g \cdot x$.
Then
$CyC = Cg \cdot xC = C(g\cdot x)g|_{x}C = CgxC = CxC$. 
Thus  $x \, \mathscr{J} \, y$.
Conversely, suppose that $x \, \mathscr{J} \, y$.
Then $y = gxh$ for some $g,h \in G$.
Thus $y = (g \cdot x)g|_{x}h$. 
But by the uniqueness of the factorization, we must have that $y = g \cdot x$.
\end{proof}

\begin{lemma} Let $C$ be a skeletal left Rees category and let $x$ be an atom where $e \stackrel{x}{\rightarrow} f$.
Choose a coset decomposition $G_{e} = \bigcup_{i \in I} g_{i} (G_{e})^{+}_{x}$. 
The set $\{g_{i}x \colon i \in I\}$ consists of pairwise $\mathscr{R}$-inequivalent elements
and every atom $\mathscr{J}$-related to $x$ is $\mathscr{R}$-related to one of these elements.
\end{lemma}
\begin{proof} Suppose that $g_{i}x \, \mathscr{R} \, g_{j}x$.
Then $g_{i}x = g_{j}xg$ for some $g \in G_{f}$.
Then $x = (g_{i}^{-1}g_{j})xg$.
But $g_{i}^{-1}g_{j} \in (G_{e})^{+}_{x}$.
By the definition of coset representatives, we have that $g_{i} = g_{j}$, as required.
Let $y \, \mathscr{J} \, x$ where $y$ is any atom.
Then $y = gxh$ for some $g,h \in G$.
Write $g = g_{i}a$ where $a \in (G_{e})^{+}_{x}$.
Then $y = g_{i}axh = g_{i}x \phi_{x}(a)h$.
Hence $y \, \mathscr{R} \, g_{i}x$.
\end{proof}

We now define a special type of basis for a left Rees category $C$.
For each hom-set $eCf$, choose a transversal $Y$ of the $\mathscr{J}$-classes of the atoms in $eCf$.
For each $x \in Y$, choose a coset decomposition $G_{e} = \bigcup_{i \in I} g_{i} (G_{e})^{+}_{x}$. 
Denote by $T_{x}^{+}$ the transversal $\{g_{i} \colon i \in I\}$.
We shall assume that the appropriate identities are always elements of these transversals.
For each such $x$, we have a set of atoms $\{kx \colon k \in T_{x}^{+} \}$
and the totality of those atoms as $x$ varies over $Y$ then provides a basis for $C$. 
We call a basis constructed in this way a {\em co-ordinatization}.
Observe that it contains two components:
a transversal of the $\mathscr{J}$-classes of the set of atoms of $C$,
we call this an {\em atomic transversal},
and for each atom $x$ in that transversal a set of coset representatives $T_{x}^{+}$.

\begin{theorem}[Skeletal left Rees categories and their diagrams]\label{the: second_theorem} \mbox{}
\begin{enumerate}

\item Let $D$ be a diagram of partial homomorphisms.
Then we may construct a skeletal left Rees category $C$ equipped with an atomic transversal such that the diagram of
partial homomorphisms constructed from the bimodule of atoms of $C$ relative to that transversal is equal to $D$. 
If $D$ is a diagram of partial isomorphisms then $C$ is a Rees category.

\item Let $C$ be a skeletal left Rees category equipped with an atomic transversal.
Then we may construct a diagram of partial homomorphisms $D$ using that transversal such that 
the tensor category of the covering bimodule associated with $D$ is isomorphic to $C$.
If $C$ is a Rees category then $D$ is a diagram of partial isomorphisms.

\end{enumerate}
\end{theorem}
\begin{proof} (1). Let $D$ be a diagram of partial homomorphisms.
Then by Lemma~\ref{le: diagram_to_bimodule}, we may construct a covering bimodule $\mathsf{B}(D)$.
This bimodule has a transversal with respect to which the associated diagram of partial homomorphisms is $D$.
By Theorem~\ref{the: first_theorem}, we may construct a left cancellative, skeletal equidivisible category
equipped with a length functor $\mathsf{T}(\mathsf{B}(D))$.
By Proposition~\ref{prop: fred} this is a skeletal left Rees category whose associated bimodule is $\mathsf{B}(D)$.
It follows by our results proved earlier that if we start with a diagram of partial isomorphisms we obtain a Rees category.

(2). Let $C$ be a skeletal left Rees category equipped with an atomic transversal.
Then by  Proposition~\ref{prop: fred} and Lemma~\ref{le: category_to_bimodule},
we may construct a covering bimodule $B$ from the set of atoms acted on by the groupoid of invertible elements.
From such a bimodule and an atomic transversal, we may construct a diagram $D$ of homomorphisms by Lemma~\ref{le: bimodule_to_diagram}.
But then by part (3) of Theorem~\ref{the: first_theorem}, 
the category $C$ is isomorphic to the tensor category $\mathsf{T}(\mathsf{B}(D))$ where $B$ is essentially $\mathsf{B}(D)$.
It follows by our results proved earlier that if we start with a Rees category we obtain a diagram of partial isomorphisms.
\end{proof}

\begin{remark}{\em We have now seen that diagrams of partial homomorphisms and skeletal left Rees categories
are two ways of viewing the same structure, the path from diagram to category taking in both the construction of bimodules
from diagrams and then tensor categories from bimodules. 
We shall now describe a {\em direct} construction of left Rees categories from diagrams 
of partial homomorphisms.}
\end{remark}

%%%%%%%%%%%%%%%%%%%%%%%%%%%%%%%%%%%%%%%%%%%%%%%%%%%%%%%%%%%%%%%%%%%%%%%%%%%%%%%
Let $D$ be a diagram of partial homomorphisms.
We shall define a category $\langle D \rangle$ by means of a presentation constructed from $D$.
We first construct a new directed graph $D'$ from $D$.
This contains the directed graph $D$ but at each vertex we adjoin additional loops
labelled by the elements of $G_{e}$.
We construct the free category $(D')^{\ast}$.
We denote elements of this category thus $x_{1} \cdot \ldots \cdot x_{m}$ where the $x_{i}$ are edges that match.
We now factor out by two kinds of relations:
those of the form $g \cdot h = gh$ where $g,h \in G_{e}$;
and those of the form $g \cdot x = x \cdot \phi_{x}(g)$ where $g \in (G_{e})_{x}^{+}$.
The resulting category is denoted by $\langle D \rangle$.

\begin{theorem}[Presentation theorem]\label{the: third_theorem} Let $D$ be a diagram of partial homomorphisms.
Then the category $\langle D \rangle$ is a skeletal left Rees category
isomorphic to the category obtained from $D$ by constructing the tensor category of the
covering bimodule associated with $D$.
The category $\langle D \rangle$ is a skeletal Rees category if and only if $D$ is a diagram of partial isomorphisms.
\end{theorem}
\begin{proof} Denote by $C$ the skeletal left Rees category constructed from the diagram 
of partial homomorphisms $D$ according to part (1) of Theorem~\ref{the: second_theorem}.
This category satisfies all the defining relations of the category $\langle D \rangle$ and so
$C$ is a functorial image of $\langle D \rangle$.

It remains therefore only to show that this functor is injective.
We work first in the category $C$.
In the category $C$ we choose an atomic transversal whose elements can be identified with the edges of the diagram $D$.
For each such atom $x$ choose a coset decomposition $G_{e} = \bigcup_{i \in I} g_{i} (G_{e})^{+}_{x}$. 
This leads to a co-ordinatization for $C$ as described earlier in this section.
We now appeal to the structure theory of Section~5, and deduce that every element of $x$ can be written
as a unique product of atoms in this co-ordinatization followed by an invertible element.
Call this a normal form.
We now work in the category $\langle D \rangle$.
Using the same coset decomposition as above, we may show that every element in $\langle D \rangle$
is equivalent in the presentation, using our two types of relations, to an element in normal forms.
However, different normal forms correspond to different elements of $C$ and so these normal forms are unique
and we have established our isomorphism.
\end{proof}

We may paraphrase the above theorem as saying that 
{\em every skeletal left Rees category may be presented by a diagram of partial homomorphisms.}

The following theorem was originally suggested by \cite{V1,V2} and \cite{Cohn}
but the proof is a straightforward generalization of Higgins's main result \cite{Higgins}
and at the same time this shows how our approach is related to his.

\begin{theorem} 
Every skeletal Rees category may be embedded in its universal groupoid.
\end{theorem}
\begin{proof} Let $C$ be a skeletal Rees category.
By Theorem~\ref{the: third_theorem}, we may assume that $C = \langle D \rangle$ where $D$ is a diagram of partial isomorphisms.
Denote by $\mathcal{G}$ the universal groupoid of $C$.
Our goal is to obtain a normal form for the elements of $\mathcal{G}$.
To that end, let $e \stackrel{x}{\rightarrow} f$.
Choose a coset decomposition $G_{e} = \bigcup_{i \in I} g_{i} (G_{e})^{+}_{x}$ to obtain the transversal $T_{x}^{+}$ as before.
However, now that we want to work in a groupoid, 
we shall also need a coset decomposition
$G_{f} = \bigcup_{j \in J} h_{j} (G_{f})^{-}_{x}$ to obtain the transversal $T_{x}^{-}$.
In both transversals, we assume that the identity elements of their respective groups have been choosen.
We may now follow the proof of the theorem in Section~3 of \cite{Higgins}
by defining suitable normal forms since finiteness plays no role in Higgins's proof.
This shows that $C$ is in fact embedded in $\mathcal{G}$.
\end{proof}

\begin{example} {\em Free monoids on $n$ generators are Rees monoids.
The group constructed according to the above theorem is the free group on $n$ generators.}
\end{example}

\begin{remark}
{\em Alternative ways of proving the above theorem are suggested by the results of Cohn \cite{Cohn} and von Karger \cite{V1,V2} though we do not pursue these here.}
\end{remark}

\begin{remark}
{\em Higgins uses the relations  $g \cdot x = x \cdot \phi_{x}(g)$ in two directions to construct the fundamental groupoid
of a diagram of partial isomorphisms.
We, on the other hand, use these relations in one direction only to construct a Rees category.
It is then evident that the Rees category sits inside the fundamental groupoid.
What is perhaps surprising is that these cancellative categories can be abstractly characterized.}  
\end{remark}

%%%%%%%%%%%%%%%%%%%%%%%%%%%%%%%%%%%%%%%%%%%%%%%%%%%%%%%%%%%%%%%%%%%%%%%%%%%%%%%%%%%%%%%%%%%%%%%%%%
\section{A categorical approach to Bass-Serre theory}

We shall now explain the connection between the theory we have developed and the theory of graphs of groups.
Our references for this theory are \cite{Serre} and \cite{W}.
We start with an observation. 
A graph of groups equipped with a given orientation is
essentially the same thing as a diagram of partial isomorphisms where the directed graph underlying it is finite
and weakly connected in the sense that as a graph it is connected.
Essentially, graphs of groups represent partial isomorphisms by means of relations and so are unoriented.
We shall call the diagrams of partial isomorphisms that arise from graphs of groups equipped with an orientation {\em Serre diagrams of partial isomorphisms}.

Let $C$ be a category.
A {\em zig-zag} joining the identity $e$ to the identity $f$ is determined by a sequence of identities
$e = e_{1}, \ldots, e_{n} = f$ such that for each consecutive pair of identities $e_{i}$ and $e_{i+1}$ we have that
either $e_{i}Ce_{i+1}$ or  $e_{i+1}Ce_{i}$ is non-empty. 
We say that a category $C$ is {\em connected} if any two identities $e,f \in C_{o}$ are joined by a zig-zag.

A skeletal Rees category $C$ is called a {\em Serre category} if it satisfies the following conditions.
\begin{description}

\item[{\rm (S1)}] The number of identities in $C$ is finite and nonzero.

\item[{\rm (S2)}] In each hom-set, the number of $\mathscr{J}$-classes of atoms is finite.

\item[{\rm (S3)}] $C$ is connected.

\end{description}

The following theorem is now immediate from the above definitions and what we proved in the previous section.

\begin{theorem}[Graphs of groups as categories]\label{the: fourth_theorem} \mbox{}

\begin{enumerate}

\item There is a correspondence between graphs of groups with a given orientation and Serre diagrams of partial isomorphisms.

\item There is a correspondence between Serre diagrams of partial isomorphisms and Serre categories.

\item The fundamental groupoid of a graph of groups with a given orientation is isomorphic to the universal groupoid
of the Serre category constructed from the diagram of partial isomorphisms associated with the oriented graph of groups.

\item The universal groupoid of a Serre category is connected.

\end{enumerate}
\end{theorem}

The following two examples are the basic building blocks of Bass-Serre theory.

\begin{examples}{\em  \mbox{}
\begin{enumerate}

\item {\bf HNN extensions}. These are constructed from Rees {\em monoids}. If we make the additional assumption that the monoid 
has the property that any two atoms are $\mathscr{J}$-related, then the universal groups are precisely HNN extensions with one stable letter.
This case was the subject of our paper \cite{LW} and motivated the work of the current paper.

\item {\bf Amalgamated free products.}  The building blocks of these are $(G,H)$-bimodules $X$ where $G$ and $H$ are both groups.
We say that such a bimodule is {\em irreducible} if there exists $x \in X$ such that $GxH = X$.
There is a bijective correspondence between conjugacy classes of partial isomorphisms from $G$ to $H$ 
and isomorphism classes of irreducible, bifree $(G,H)$-bisets.
Consider now any irreducible, bifree biset $(G,X,H)$.
Choose and fix $x \in X$.
Let $A = G_{x}^{+}$, $B = H_{x}^{-}$ and $\theta = \phi_{x}$ be the associated partial isomorphism.
We may regard $(G,X,H)$ as a cancellative category with two identities in the following way.
We let the identity of $G$, $1_{G}$ say, be one of the identities and the identity of $H$, $1_{H}$ say, the other.
Thus we take the disjoint union $G \cup X \cup H$.
The products in $G$ and $H$ are the group products.
The product $gx$ is the action and the product $xh$ is the action.
We denote by $C$ the above biset regarded as a category in this way.
Then $C$ is a Rees category and is in fact a Serre category.
In this case, the tensor product cosntruction is essentially degenerate.
The universal groupoid of $C$ is connected and any vertex group is isomorphic to  $G \ast_{\theta} H$
the amalgamated free product of $G$ and $H$ via the identifying partial isomorphism $\theta$.

\end{enumerate}
}
\end{examples}

%%%%%%%%%%%%%%%%%%%%%%%%%%%%%%%%%%%%%%%%%%%%%%%%%%%%%%%%%%%%%%%%%%%%%%%%%%%%%%%%%%%%%%
\subsection{Concluding remarks}

Andrew Duncan (Newcastle, UK) has pointed out that there are some interesting parallels between our theory and Stallings theory of pregroups
as described in \cite{Hoare}.
We do not know, at this point, whether we can derive that theory from ours.

The Serre tree of a graph of groups may be constructed from the way the associated Rees category is embedded in its universal group \cite{W}.
This is strongly reminiscent of the categorical interpretation of the proof of McAlister's $P$-theorem described in \cite{JamesL}.
In fact, using \cite{JL2}, every Rees category gives rise to an inverse semigroup which is strongly $E^{\ast}$-unitary.
The maximum enlargement theorem, described in Chapter~8 of \cite{Lawson1998}, may then be used to construct the Serre tree.
Nick Gilbert (Heriot-Watt, UK) has a different approach to Bass-Serre theory \cite{G} which uses ordered groupoids.
We do not yet know how his approach can be reconciled with the one we have sketched above.

All cancellative equidivisible categories may be embedded in groupoids; see \cite{V1} although a direct proof using the ideas of \cite{PJ} would be useful.
Motivated by group theory, one might consider the structure of such categories equipped with other kinds of length functors.

%%%%%%%%%%%%%%%%%%%%%%%%%%%%%%%%%%%%%%%%%%%%%%%%%%%%%%%%%%%%%%%%%%%%%%%%%%%%%%%%%%%%%%%%

\end{document}